\documentclass[a4paper,11pt]{article}
\usepackage{amsmath,amsthm,amssymb,epic,eepicemu}
\usepackage[cp1251]{inputenc}
\usepackage{graphicx,psfrag}
\usepackage[english]{babel}
\usepackage[pdftex]{hyperref}

\def\Mat{\mathop{\rm Mat}\nolimits}
\def\Integer{\mathbb{Z}}
\def\bbG{\mathbb{G}}
\def\bbP{\mathbb{P}}
\def\bbR{\mathbb{R}}

\def\P{\mathbb{P}}
\def\ee{\mathfrak{e}}
\def\SP<#1>{\langle#1\rangle}

\theoremstyle{plain}
\newtheorem{theorem}{Theorem}
\newtheorem{lemma}{Lemma}

\theoremstyle{definition}
\newtheorem{definition}{Definition}
\theoremstyle{remark}
\newtheorem{example}{Example}

\title{Integrable discrete nets in Grassmannians
}
\author{Vsevolod E. Adler$^{1,2}$, Alexander I. Bobenko$^2$, Yuri B. Suris$^3$}
\date{\empty}
\begin{document} \maketitle

\footnotetext[1]{L.D. Landau Institute for Theoretical Physics, 1a Semenov pr.,
 142432 Chernogolovka, Russia. E-mail: adler@itp.ac.ru}
\footnotetext[2]{Institut f\"ur Mathematik, Technische
 Universit\"at Berlin, Str. des 17. Juni 135, 10623 Berlin, Germany.
 E-mail: bobenko@math.tu-berlin.de}
\footnotetext[3]{Zentrum Mathematik, Technische Universit\"at
 M\"unchen, Boltzmannstr. 3, 85748 Garching, Germany. E-mail:
 suris@ma.tum.de}

\begin{abstract}

We consider discrete nets in Grassmannians $\bbG^{d}_{r}$ which
generalize Q-nets (maps $\Integer^N\to\bbP^d$ with planar
elementary quadrilaterals) and Darboux nets ($\bbP^d$-valued maps
defined on the edges of $\Integer^N$ such that quadruples of
points corresponding to elementary squares are all collinear). We
give a geometric proof of integrability (multidimensional
consistency) of these novel nets, and show that they are
analytically described by the noncommutative discrete Darboux
system.

\medskip

\end{abstract}

\section{Introduction}\label{s:intro}

One of the central notions in discrete differential geometry
constitute discrete nets, that is, maps $\Integer^3\to\P^d$
specified by certain geometric properties. Their study was
initiated by R.~Sauer \cite{Sauer1}, while their appearance in the
modern theory of integrable systems is connected with the work of
A.~Bobenko and U.~Pinkall
\cite{Bobenko_Pinkall_1996a,Bobenko_Pinkall_1996b} and of
A.~Doliwa and P.~Santini \cite{Doliwa_Santini_1997}. A systematic
exposition of discrete differential geometry, including detailed
bibliographical and historical remarks, is given in the monograph
\cite{Bobenko_Suris_2005} by two of the present authors. In many
aspects, the discrete differential geometry of parametrized
surfaces and coordinate systems turns out to be more transparent
and fundamental than the classical (smooth) differential geometry,
since the transformations of discrete surfaces possess the same
geometric properties and therefore are described by the same
equations as the surfaces themselves. This leads to the notion of
multidimensional consistency which can be seen as the fundamental
geometric definition of integrability in the discrete context,
which yields standard integrability structures of both discrete
and continuous systems, such as B\"acklund and Darboux
transformations, zero curvature representations, hierarchies of
commuting flows etc.

In this note we present a generalization of two classes of
multidimensional nets, Q-nets (or discrete conjugate nets) and
Darboux nets, to the maps with values in the Grassmannian
$\bbG^{d}_{r}$ instead of $\P^d$. The basic idea underlying this
work goes back to H.~Grassmann and J.~Pl\"ucker and consists in
regarding more complicated objects than just points (like lines,
spheres, multidimensional planes, contact elements etc.) as
elementary objects of certain geometries. Such objects are then
represented as points belonging to some auxiliary projective
spaces or to certain varieties in these spaces. In the framework
of discrete differential geometry, one can assign such objects to
the sites of the lattice $\Integer^N$ and impose certain geometric
conditions to characterize interesting classes of multidimensional
nets. Several such classes have been introduced in the literature,
for instance:
\begin{itemize}
 \item discrete line congruences \cite{Doliwa_Santini_Manas}, which
are nets in the set of lines in $\bbP^d$ subject to the condition
that any two neighboring lines intersect (are coplanar);
 \item discrete W-congruences \cite{Doliwa_2001}, which are nets in
the set of lines in $\bbP^3$ such that four lines corresponding to
the vertices of every elementary square of $\Integer^N$ belong to
a regulus. If one represents the lines in $\bbP^3$ by points of
the Pl\"ucker quadric in $\bbP(\bbR^{4,2})$, then this condition
is equivalent to the planarity of elementary quadrilaterals;
 \item discrete R-congruences of spheres \cite{Doliwa_2001b, Bobenko_Suris_2007a},
which are nets in the set of oriented spheres in $\bbR^3$, which
in the framework of Lie geometry are represented by points of the
Lie quadric in $\bbP(\bbR^{3,3})$. Again, the defining condition
is the planarity of all elementary quadrilaterals of the net;
 \item principal contact element nets \cite{Bobenko_Suris_2007a},
which are nets in the set of contact elements in $\bbR^3$ such
that any two neighboring contact elements share a common oriented
sphere. In the framework of Lie geometry such nets are represented
by isotropic line congruences.
\end{itemize}
Again, one can find detailed information and additional
bibliographical notes about these nets in
\cite{Bobenko_Suris_2005}.

In the present work, we study two related classes of
multidimensional nets in Grassmannians $\bbG^d_r$ which generalize
Q-nets (nets in $\bbP^d$ with planar elementary quadrilaterals)
and the so called Darboux nets introduced in \cite{Schief_2003}.

It turns out that Grassmannian Q-nets can be analytically
described by a noncommutative version of the so called discrete
Darboux system which was introduced, without a geometric
interpretation, in \cite{Bogdanov_Konopelchenko_1995}. Our present
investigations provide also a geometric meaning for the abstract
Q-nets in a projective space over a noncommutative ring,
considered in \cite{Doliwa_2008}. More precisely, we demonstrate
that equations of abstract Q-nets in a projective space over the
matrix ring $\Mat(r+1,r+1)$ can be interpreted as the analytical
description of the Grassmannian Q-nets in the suitable
parametrization. The fact that the equations of Q-nets are
considered over a ring rather than over a field is not very
essential in this context, since the very notion of Q-nets is
related to subspaces in general position, and an accident
degeneration of some coefficients is treated as a singularity of
the discrete mapping. A much more important circumstance is the
noncommutativity of the matrix ring, which is equivalent to the
absence of Pappus theorem in the geometries over this ring.

The main results of the paper are Theorems \ref{th:Gr_3D},
\ref{th:Gr_4D} on the multidimensional consistency of Grassmannian
Q-nets and Theorem \ref{th:bij} on the equations for integrable
evolution of the discrete rotation coefficients.

\section{Multidimensional consistency of Grassmannian Q-nets}\label{s:def}

Recall that the Grassmannian $\bbG^{d}_{r}$ is defined as the
variety of $r$-planes in $\bbP^d$. It can be also described as the
variety of $(r+1)$-dimensional vector subspaces of the
$(d+1)$-dimensional vector space $\bbR^{d+1}$. In the latter
realization, the Grassmannian is alternatively denoted by
$G^{d+1}_{r+1}$. In what follows, the term ``dimension'' is used
in the projective sense.

\begin{definition}\label{def:Gr_lattice}{\bf (Grassmannian Q-net)}
A map $\Integer^N\to\bbG^{d}_{r}$, $N\ge 2$, $d>3r+2$, is called a
$N$-dimensional Grassmannian Q-net of rank $r$ if for every
elementary square of $\Integer^N$ the four $r$-planes in $\P^d$
corresponding to its vertices belong to some $(3r+2)$-plane.
\end{definition}

Note that three generic $r$-planes in $\P^d$ span a
$(3r+2)$-plane. Therefore, the meaning of Definition
\ref{def:Gr_lattice} is that if any three of the $r$-planes
corresponding to an elementary cell are chosen in general
position, then the last one belongs to the $(3r+2)$-plane spanned
by the first three.

\begin{example}
In the case of rank $r=0$ Definition \ref{def:Gr_lattice} requires
that four points corresponding to any elementary square of
$\Integer^N$ be coplanar. Thus we arrive at the notion of usual
Q-nets.
\end{example}
\begin{example}
Q-nets of rank $r=1$ are built of projective lines assigned to
vertices of the lattice $\Integer^N$, and Definition
\ref{def:Gr_lattice} requires that four lines corresponding to any
elementary square lie in a 5-plane.
\end{example}

The main properties of usual Q-nets which will be generalized now
to the Grassmannian context are the following (see a detailed
account in \cite{Bobenko_Suris_2005}):
\begin{itemize}
\item Within an elementary cube of $\Integer^3$, the points
assigned arbitrarily to any seven vertices determine the point
assigned to the eighth vertex uniquely. This can be expressed by
saying that Q-nets are described by a {\em discrete 3D system}.
\item This 3D system can be imposed on all 3D faces of an
elementary cube of any dimension $N\ge 4$. This property is called
the {\em multidimensional consistency} of the corresponding 3D
system and follows for any $N\ge 4$ from the 4D consistency. The
multidimensional consistency is treated as the integrability of
the corresponding 3D system.
\end{itemize}
These properties are illustrated on Fig. \ref{fig:3D}. On this
figure and everywhere else we use the notation $X_i$ for the shift
of the $i$-th argument of a function $X$ on $\Integer^N$, that is,
for $X(n_1,\dots,n_i+1,\dots)$. It is clear that the order of the
subscripts does not matter, $X_{ij}=X_{ji}$.

\begin{figure}
\setlength{\unitlength}{0.04em}
\begin{picture}(300,240)(-30,0)
 \put(0,0){\circle*{15}}    \put(150,0){\circle*{15}}
 \put(0,150){\circle*{15}}  \put(150,150){\circle*{15}}
 \put(50,50){\circle*{15}} \put(50,200){\circle*{15}}
 \put(200,50){\circle*{15}}
 \put(200,200){\circle{15}}
 \path(0,0)(150,0)       \path(0,0)(0,150)
 \path(150,0)(150,150)   \path(0,150)(150,150)
 \path(0,150)(50,200)    \path(150,150)(194,194)
 \path(50,200)(192.5,200)
 \path(200,192.5)(200,50) \path(200,50)(150,0)
 \dashline[+30]{10}(0,0)(50,50)
 \dashline[+30]{10}(50,50)(50,200)
 \dashline[+30]{10}(50,50)(200,50)
 \put(-30,-5){$X$}
 \put(-40,145){$X_3$} \put(215,45){$X_{12}$}
 \put(165,-5){$X_1$} \put(160,140){$X_{13}$}
 \put(10,50){$X_2$}  \put(5,205){$X_{23}$}
 \put(215,200){$X_{123}$}
\end{picture}
\hfill \setlength{\unitlength}{0.07em}
\begin{picture}(200,220)(-90,-90)

 \drawline(15,-20)(50,0)(50,47)
 \drawline(47,50)(0,50)(-35,30)(-35,-20)(15,-20)(15,30)(47,48.5)
 \drawline(15,30)(-35,30)
 \dashline{4}(-35,-20)(0,0)(0,50)\dashline{4}(0,0)(50,0)
 \drawline(30,-90)(131,-32)
 \drawline(135,-26)(135,116)
 \drawline(131,120)(-11,120)
 \drawline(-19,118)(-120,60)(-120,-90)(30,-90)(30,56)
 \drawline(34,62)(131,118)
 \drawline(26,60)(-120,60)
 \dashline{4}(-120,-90)(-15,-30)(-15,116)
 \dashline{4}(-15,-30)(131,-30)
  \dashline{2}(0,0)(-15,-30)
  \dashline{2}(-35,-20)(-120,-90)
  \dashline{2}(50,0)(132,-29)
  \dashline{2}(0,50)(-14,116)
  \dashline{2}(15,-20)(30,-90)
  \dashline{2}(-35,30)(-120,60)
  \dashline{2}(53,51.5)(131,117.5)
  \dashline{2}(15,30)(26,56)

  \put(-35,-20){\circle*{8}}       
  \put(15,-20){\circle*{8}}        
  \put(0,0){\circle*{8}}           
  \put(-35,30){\circle*{8}}        
  \put(50,0){\circle*{8}}          
  \put(0,50){\circle*{8}}          
  \put(15,30){\circle*{8}}         
  \put(50,50){\circle{8}}          

  \put(-120,-90){\circle*{10}}     
  \put(-15,-30){\circle*{9}}       
  \put(30,-90){\circle*{10}}       
  \put(-120,60){\circle*{10}}      
  \put(135,-30){\circle{10}}       
  \put(-15,120){\circle{10}}       
  \put(30,60){\circle{10}}         
  \put(130,115){$\Box$}            

  \put(-50,-15){$X$}
  \put(-3,-15){$X_1$}
  \put(-23,6){$X_2$}
  \put(-55,20){$X_3$}
  \put(58,3){$X_{12}$}
  \put(5,15){$X_{13}$}
  \put(-30,50){$X_{23}$}
  \put(58,44){$X_{123}$}
  \put(-150,-90){$X_4$}
  \put(0,-83){$X_{14}$}
  \put(-12,-43){$X_{24}$}
  \put(-150,55){$X_{34}$}
  \put(145,-28){$X_{124}$}
  \put(10,71){$X_{134}$}
  \put(-9,103){$X_{234}$}
  \put(140,103){$X_{1234}$}
\end{picture}
\hfill\strut \caption{The combinatorial meaning of a discrete 3D
system and its 4D consistency. Black circles mark the initial data
within an elementary cube; white circles mark the vertices
uniquely determined by the initial data; white square marks the
vertex where the consistency condition appears. In Grassmannian
Q-nets the vertices carry $r$-planes; two $r$-planes corresponding
to an edge span a $(2r+1)$-plane; four $r$-planes assigned to the
vertices of an elementary square span a $(3r+2)$-planes; eight
$r$-planes assigned to the vertices of an elementary 3-cube span a
$(4r+3)$-plane in $\P^d$.} \label{fig:3D}
\end{figure}
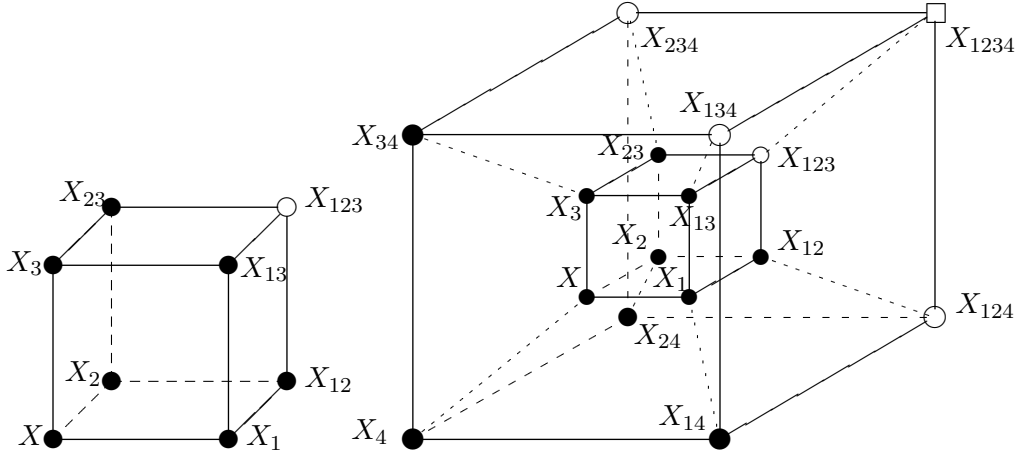

\begin{theorem}\label{th:Gr_3D}
{\bf (Grassmannian Q-nets are described by a discrete 3D system)}
Let seven $r$-planes $X$, $X_i$, $X_{ij}\in\bbG^d_r$, $1\le i\ne
j\le3$, $d\ge4r+3$, be given such that
\[
 \dim {\rm span}(X,X_i,X_j,X_{ij})=3r+2
\]
for each pair of indices $ij$, but with no other degeneracies.
Then there exists a unique $r$-plane $X_{123}\in\bbG^d_r$ such
that the conditions
\[
 \dim {\rm span}(X_i,X_{ij},X_{ik},X_{123})=3r+2
\]
are fulfilled as well.
\end{theorem}
\begin{proof}
The general position condition yields that the projective plane
$V={\rm span}(X,X_1,X_2,X_3)$ is of dimension $4r+3$. The
assumptions of the theorem imply that the $r$-planes $X_{ij}$ are
contained in the corresponding $(3r+2)$-planes ${\rm
span}(X,X_i,X_j)$, and therefore are also contained in $V$. In the
case of the general position, the planes spanned by
$X_i,X_{ij},X_{ik}$ are also $(3r+2)$-dimensional. The $r$-plane
$X_{123}$, if exists, must lie in the intersection of three such
$(3r+2)$-planes. In the $(4r+3)$-dimensional space $V$, the
dimension of a pairwise intersection is $2(3r+2)-(4r+3)=2r+1$, and
therefore the dimension of the triple intersection is
$(4r+3)-3(3r+2)+3(2r+1)=r$, as required.
\end{proof}

\begin{theorem}\label{th:Gr_4D}{\bf (Multidimensional consistency of Grassmannian
Q-nets)} The 3D system governing Grassmannian Q-nets is
4D-consistent and therefore $N$-dimensionally consistent for all
$N\ge 4$.
\end{theorem}
\begin{proof}
One has to show that the four $r$-planes,
\[
{\rm span}(X_{12},X_{123},X_{124})\cap\,{\rm
span}(X_{13},X_{123},X_{124})\cap\,{\rm span}(X_{14},
X_{124},X_{134}),
\]
and the three others obtained by cyclic shifts of indices,
coincide. Thus, we have to prove that the six $(3r+2)$-planes
${\rm span}(X_{ij},X_{ijk},X_{ij\ell})$ intersect along a
$r$-plane. We assume that the ambient space $\bbP^d$ has dimension
$d\ge 5r+4$. Then, in general position, the plane ${\rm
span}(X,X_1,X_2,X_3,X_4)$ which contains all elements of our
construction is of dimension $5r+4$. It is easy to understand that
the $(3r+2)$-plane ${\rm span}(X_{ij},X_{ijk},X_{ij\ell})$ is the
intersection of two $(4r+3)$-planes $V_i={\rm
span}(X_{i},X_{ij},X_{ik},X_{i\ell})$ and $V_j={\rm
span}(X_{j},X_{ij},X_{jk},X_{j\ell})$. Indeed, the plane $V_i$
contains also $X_{ijk}$, $X_{ij\ell}$, and $X_{ik\ell}$.
Therefore, both $V_i$ and $V_j$ contain the three $r$-planes
$X_{ij}$, $X_{ijk}$ and $X_{ij\ell}$, which determine the
$(3r+2)$-plane ${\rm span}(X_{ij},X_{ijk},X_{ij\ell})$. Now the
intersection in question can be alternatively described as the
intersection of the four $(4r+3)$-planes $V_1$, $V_2$, $V_3$,
$V_4$ of one and the same $(5r+4)$-dimensional space. This
intersection is generically an $r$-plane.
\end{proof}

\section{Analytical description: noncommutative Q-nets}\label{s:noncom}

Here we give an analytical description of Grassmannian Q-nets. In
the case of ordinary Q-nets (of rank $r=0$), the planarity
condition is written in affine coordinates as
\begin{equation}\label{xijack}
 x_{ij}=x+a^{ij}(x_i-x)+a^{ji}(x_j-x),
\end{equation}
where the scalar coefficients $a^{ij},a^{ji}$ are naturally
assigned to the corresponding elementary squares of $\Integer^N$
(parallel to the coordinate plane $(ij)$). Consistency of these
equations around an elementary cube (Theorem \ref{th:Gr_3D})
yields a mapping
\[
 (a^{12},a^{21},a^{13},a^{31},a^{23},a^{32})\mapsto
 (a^{12}_3,a^{21}_3,a^{13}_2,a^{31}_2,a^{23}_1,a^{32}_1).
\]
This mapping can be rewritten in a rather nice form in terms of so
called rotation coefficients. The same approach works in the case
$r>0$ as well, with the only difference that now we have to assume
that the coefficients $a^{ij}$ are noncommutative.

In order to demonstrate this we use the interpretation of the
Grassmannian $\bbG^d_r$ as the variety $G^{d+1}_{r+1}$ of all
$(r+1)$-dimensional subspaces of the vector space $\bbR^{d+1}$.
One can represent an $(r+1)$-dimensional subspace $X$ of
$\bbR^{d+1}$ by a $(r+1)\times(d+1)$ matrix $x$ whose rows contain
vectors of some basis of $X$. The change of basis of $X$
corresponds to a left multiplication of $x$ by an element of
$GL_{r+1}$. Thus, one gets the isomorphism $G^{d+1}_{r+1}\equiv
(\bbR^{d+1})^{r+1}/GL_{r+1}$.

The condition that the $(r+1)$-dimensional vector subspace
$X_{12}$ belongs to the $3(r+1)$-dimensional vector space spanned
by $X,X_1,X_2$ is now expressed by an equation
\[
 x_{12}=ax+bx_1+cx_2,\quad a,b,c\in\Mat(r+1,r+1),
\]
The set of coefficients $a,b,c$ is abundant since it contains
$3(r+1)^2$ parameters while $\dim G^{3r+3}_{r+1}=2(r+1)^2$. In
order to get rid of this abundance we adopt an ``affine''
normalization of the representatives $x$, analogous to the case
$r=0$. Namely, the representative of a generic subspace can be
chosen, by applying the left multiplication by a suitable matrix,
in the form
\begin{equation}\label{xaff}
 x=\begin{pmatrix}
  x^{1,1}   &\dots & x^{1,d-r}   & 1 & \dots  & 0\\
  \vdots    &      & \vdots      &   & \ddots & \\
  x^{r+1,1} &\dots & x^{r+1,d-r} & 0 & \dots  & 1
 \end{pmatrix}
\end{equation}
with the unit matrix $I$ in the last $r+1$ columns. Under this
normalization the coefficients in the equation
$x_{12}=ax+bx_1+cx_2$ obey the relation $I=a+b+c$, and we come to
the equation of the form (\ref{xijack}).

The calculation of the consistency conditions of the equations
(\ref{xijack}) remains rather simple in the noncommutative setup. One of
three ways of getting vector $x_{ijk}$ is
\[
 x_{ijk}=x_k+a^{ij}_k(x_{ik}-x_k)+a^{ji}_k(x_{jk}-x_k).
\]
Note that after we substitute $x_{ik}$ and $x_{jk}$ from
(\ref{xijack}), the matrix $x_i$ enters the right hand side only
once with the coefficient $a^{ij}_ka^{ik}$. Therefore,
alternating of $j$ and $k$ yields the relation
\begin{equation}\label{aa}
 a^{ij}_ka^{ik}=a^{ik}_ja^{ij}.
\end{equation}
Analysis of relations \eqref{aa} is based on the following
statement.
\begin{lemma}\label{lemma closed}
{\bf (Integration of closed multiplicative matrix-valued one-form)}
Let the $GL_{r+1}$-valued functions $a^j$ be defined on edges of
$\Integer^N$ parallel to the $j$-th coordinate axis (so that
$a^j(n)$ is assigned to the edge $[n,n+e_j]$, where $e_j$ is the
unit vector of the $j$-th coordinate axis). If $a^j$ satisfy
\begin{equation}\label{eq:closed}
 a^j(n+e_k)a^k(n)=a^k(n+e_j)a^j(n),
\end{equation}
then there exists a $GL_{r+1}$-valued function $h$ defined on
vertices of $\Integer^N$ such that
\begin{equation}\label{eq:factorized}
 a^j(n)=h(n+e_j)h^{-1}(n),\quad 1\le j\le N.
\end{equation}
\end{lemma}
\begin{proof}
Prescribe $h(0)$ arbitrarily. In order to extend $h$ to any point
of $\Integer^N$, connect it to $0$ by a lattice path
$(\ee_1,\ldots,\ee_M)$, where the endpoint of any edge $\ee_m$
coincides with the initial point of the edge $\ee_{m+1}$. Extend
$h$ along the path according to \eqref{eq:factorized}. This
extension does not depend on the choice of the path. Indeed, any
two lattice paths connecting any two points can be transformed
into one another by means of elementary flips exchanging two edges
$[n,n+e_j]$, $[n+e_j,n+e_j+e_k]$ to the two edges $[n,n+e_k]$,
$[n+e_k,n+e_j+e_k]$. The value of $h$ at the common points $n$ and
$n+e_j+e_k$ of such two paths remain unchanged under the flip, as
follows from the ``closedness condition'' \eqref{eq:closed}.
\end{proof}

Equations \eqref{aa} together with Lemma \ref{lemma closed} yield
existence of matrices $h^i$ (assigned to edges of $\Integer^N$
parallel to the $i$-th coordinate axis) such that
$a^{ij}=h^i_j(h^i)^{-1}$. They are called the {\em discrete Lam\'e
coefficients}. Equation (\ref{xijack}) takes the form
\[
 x_{ij}=x+h^i_j(h^i)^{-1}(x_i-x)+h^j_i(h^j)^{-1}(x_j-x).
\]
Let us introduce the new variable $y^i$ (also assigned to edges
parallel to the $i$-th coordinate axis) by the formula
$x_i-x=h^iy^i$. Then $x_{ij}=x+h^i_jy^i+h^j_iy^j$, and, on the
other hand, $x_{ij}=x_j+h^i_jy^i_j=x+h^jy^j+h^i_jy^i_j$. This
allows to rewrite the equation (\ref{xijack}) finally as
\begin{equation}\label{yij}
 y^i_j=y^i-b^{ij}y^j.
\end{equation}
The matrices
\begin{equation}\label{bij}
b^{ij}=(h^i_j)^{-1}(h^j_i-h^j)
\end{equation}
(assigned to the elementary squares parallel to the $(ij)$-th
coordinate plane) are called the {\em discrete rotation
coefficients}. The compatibility conditions in terms of these
coefficients are perfectly simple. We have
\[
 y^i_{jk}=y^i+b^{ik}y^k+b^{ij}_k(y^j+b^{jk}y^k)
  =y^i+b^{ij}y^j+b^{ik}_j(y^k+b^{kj}y^j)
\]
which leads to the coupled equations
\[
  b^{ij}_k-b^{ik}_jb^{kj}=b^{ij},\quad
  -b^{ij}_kb^{jk}+b^{ik}_j=b^{ik}.
\]
These can be solved for $b^{ij}_k$ to give an explicit map.

\begin{theorem}\label{th:bij}{\bf (Grassmannian Q-nets are
described by the noncommutative discrete Darboux system)} Rotation
coefficients $b^{ij}\in\Mat(r+1,r+1)$ of Q-nets in the
Grassmannian $\bbG^d_r$ satisfy the noncommutative discrete
Darboux system
\begin{equation}\label{bij map}
 b^{ij}_k=(b^{ij}+b^{ik}b^{kj})(I-b^{jk}b^{kj})^{-1}, \quad
k\neq i\neq j\neq k.
\end{equation}
This map is multidimensionally consistent.
\end{theorem}

Consistency is a corollary of Theorem \ref{th:Gr_4D}, but it is
also not too difficult to prove it directly.


\section{Grassmannian Darboux nets}

\begin{definition}{\bf (Grassmannian Darboux net)}
A Grassmannian Darboux net (of rank $r$) is a map
$E(\Integer^N)\to \bbG^{d}_{r}$ defined on edges of the regular
square lattice, such that for every elementary quadrilateral of
$\Integer^N$ the four $r$-planes corresponding to its sides lie in
a $(2r+1)$-plane.
\end{definition}

In particular, for $r=0$ one arrives at the notion of Darboux nets
introduced in \cite{Schief_2003}: the four points corresponding to
the sides of every elementary square are required to be collinear.

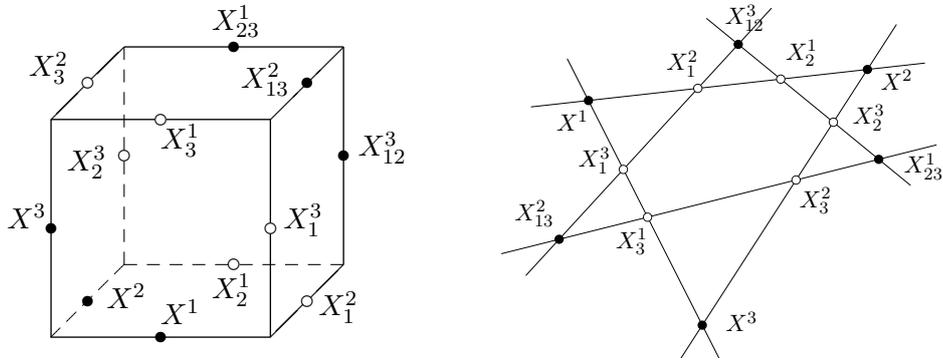
\begin{figure}[htbp]
\setlength{\unitlength}{0.05em}
\begin{picture}(200,240)(-50,-50)
 \put(75,0){\circle*{8}}
 \put(25,25){\circle*{8}}
 \put(0,75){\circle*{8}}
 \put(200,125){\circle*{8}}
 \put(175,175){\circle*{8}}
 \put(125,200){\circle*{8}}
 \put(150,75){\circle{8}}
 \put(125,50){\circle{8}}
 \put(50,125){\circle{8}}
 \put(75,150){\circle{8}}
 \put(25,175){\circle{8}}
 \put(175,25){\circle{8}}
 \path(0,0)(150,0)
 \path(0,0)(0,150)
 \path(0,150)(22.2,172.2)\path(27.8,177.8)(50,200)
 \path(150,0)(150,71)\path(150,79)(150,150)
 \path(0,150)(71,150)\path(79,150)(150,150)
 \path(150,150)(200,200)(50,200)
 \path(200,200)(200,50)
 \path(150,0)(172.2,22.2)\path(177.8,27.8)(200,50)
 \dashline[+30]{10}(0,0)(50,50)
 \dashline[+30]{10}(50,50)(50,121)
 \dashline[+30]{10}(50,129)(50,200)
 \dashline[+30]{10}(50,50)(121,50)
 \dashline[+30]{10}(129,50)(200,50)
 \put(75,9){$X^1$} \put(38,20){$X^2$} \put(-30,75){$X^3$}
 \put(130,172){$X_{13}^2$}
 \put(110,212){$X_{23}^1$}
 \put(210,125){$X_{12}^3$}
 \put(110,27){$X_2^1$}  \put(183,15){$X_1^2$}
 \put(-15,180){$X_3^2$} \put(11,115){$X_2^3$}
 \put(159,75){$X_1^3$}  \put(75,130){$X_3^1$}
\end{picture}\hspace{1cm}
 \includegraphics[width=0.65\textwidth]{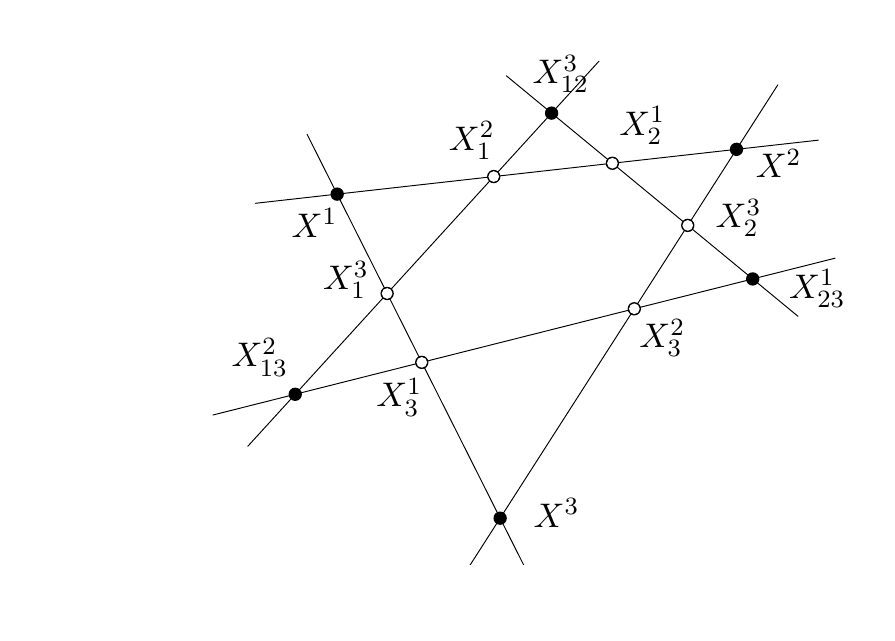}
\caption{Combinatorics (left) and geometry (right) of an
elementary cube of a Darboux net: black circles mark the initial
data; white circles mark the data uniquely determined by the
initial data. In Grassmannian Darboux nets the edges carry
$r$-planes; four $r$-planes corresponding to an elementary square
span a $(2r+1)$-plane; twelve $r$-planes assigned to the edges of
an elementary 3-cube span a $(3r+2)$-plane in
$\bbP^d$.}\label{Fig: schief system}
\end{figure}

We will denote by $X^i$ the $r$-planes assigned to the edges of
$\Integer^N$ parallel to the $i$-th coordinate axis; the
subscripts will be still reserved for the shift operation.

To find an analytical description of Grassmannian Darboux nets, we
continue to work with the ``affine'' representatives from
$G^{d+1}_{r+1}$ normalized as in \eqref{xaff}. The defining
property yields:
\[
  x^i_j=r^{ij}x^i+(I-r^{ij})x^j.
\]
Hence
\[
 x^i_{jk}=r^{ij}_k(r^{ik}x^i+(I-r^{ik})x^k)
  +(I-r^{ij}_k)(r^{jk}x^j+(I-r^{jk})x^k),
\]
and therefore
\[
 r^{ij}_kr^{ik}=r^{ik}_jr^{ij}.
\]
Comparing with \eqref{aa} and using Lemma \ref{lemma closed}, we
conclude that $r^{ij}=s^i_j(s^i)^{-1}$. Set $y^i=(s^i)^{-1}x^i$,
then the linear problem takes the form \eqref{yij} with the
rotation coefficients
\[
 b^{ij}=((s^i)^{-1}-(s^i_j)^{-1})s^j.
\]
Thus, we come to the conclusion that Grassmannian Darboux nets are
described by the same noncommutative discrete Darboux system
\eqref{bij map} as Grassmannian Q-nets, with rotation coefficients
$b^{ij}\in\Mat(r+1,r+1)$ defined by the last formula. Of course,
this is not a coincidence, since Q-nets and Darboux nets are
closely related. Indeed, considering an intersection of a
Grassmannian Q-net in $\P^d$ with some plane $\Pi$ of codimension
$r+1$, one will find a Grassmannian Darboux net in $\Pi$.
Conversely, any Grassmannian Darboux net can be extended
(non-uniquely) to a Grassmannian Q-net. This is analogous to the
case of ordinary nets (of rank $r=0$) explained in
\cite[p.76]{Bobenko_Suris_2005}.

\paragraph{Acknowledgements.} The research of V.A. is supported by the
DFG Research Unit ``Polyhedral Surfaces'' and by RFBR grants
08-01-00453 and NSh-3472.2008.2. The research of A.B. is partially
supported by the DFG Research Unit ``Polyhedral Surfaces''.


\end{document}